\documentclass{article}
\usepackage{amsfonts}
\usepackage{amsmath}

\newtheorem{theorem}{Theorem}

\newtheorem{lemma}[theorem]{Lemma}

\newtheorem{proposition}[theorem]{Proposition}
\newtheorem{remark}[theorem]{Remark}

\newenvironment{proof}[1][Proof]{\noindent\textbf{#1.} }{\ \rule{0.5em}{0.5em}}
\input{tcilatex}

\begin{document}

\title{On the Order of Magnitude of Sums of Negative Powers of Integrated
Processes\thanks{%
I would like to thank Kalidas Jana for inquiring about the order of
magnitude of some of the quantities now treated in the paper. I am indebted
to Robert de Jong for comments on an early draft that have led to an
improvement in Theorem \ref{mainresult}. I am grateful to Istvan Berkes,
Hannes Leeb, David Preinerstorfer, Zhan Shi, the referees, and the editor
Peter Phillips for helpful comments.}}
\author{Benedikt M. P\"{o}tscher \\
University of Vienna}
\date{Preliminary Draft November 2010\\
First version: January 2011\\
First revision: December 2011\\
This version: January 2012}
\maketitle

\section{Introduction}

The asymptotic behavior of expressions of the form $%
\sum_{t=1}^{n}f(r_{n}x_{t})$ where $x_{t}$ is an integrated process, $r_{n}$
is a sequence of norming constants, and $f$ is a measurable function has
been the subject of a number of articles in recent years. We mention Borodin
and Ibragimov (1995), Park and Phillips (1999), de Jong (2004), Jeganathan
(2004), P\"{o}tscher (2004), de Jong and Whang (2005), Berkes and Horvath
(2006), and Christopeit (2009) which study weak convergence results for such
expressions under various conditions on $x_{t}$ and the function $f$. Of
course, these results also provide information on the order of magnitude of $%
\sum_{t=1}^{n}f(r_{n}x_{t})$. However, to the best of our knowledge no
result is available for the case where $f$ is non-integrable with respect to
Lebesgue-measure in a neighborhood of a given point, say $x=0$. In this
paper we are interested in bounds on the order of magnitude of $%
\sum_{t=1}^{n}\left\vert x_{t}\right\vert ^{-\alpha }$ when $\alpha \geq 1$,
a case where the implied function $f$ is not integrable in any neighborhood
of zero. More generally, we shall also obtain bounds on the order of
magnitude for $\sum_{t=1}^{n}v_{t}\left\vert x_{t}\right\vert ^{-\alpha }$
where $v_{t}$ are random variables satisfying certain conditions. While the
emphasis in this paper is on negative powers that are non-integrable in any
neighborhood of zero (i.e., $\alpha \geq 1$), we also present results for $%
\alpha <1\mathbb{\ }$ whenever they are easily obtained. We make no effort
to improve the results in case $\alpha <1$, but we shall occasionally
mention better results available in this case (or in subcases thereof)
without attempting to be complete in the coverage of such (better) results
specific to the case $\alpha <1$. While my interest in the problem treated
in the present paper is purely driven by mathematical curiosity, reciprocals
and ratios of variables that may be integrated are not alien to economic
models. Hence the results presented below are of potential interest for the
econometric analysis of such models.

\section{Results}

Consider an integrated process%
\begin{equation*}
x_{t}=x_{t-1}+w_{t}
\end{equation*}%
for integer $t\geq 1$, with the initial real-valued random variable $x_{0}$
being independent of the process $(w_{t})_{t\geq 1}$ which is assumed to be
given by%
\begin{equation*}
w_{t}=\sum_{j=0}^{\infty }\phi _{j}\varepsilon _{t-j}.
\end{equation*}%
Here $(\varepsilon _{i})_{i\in \mathbb{Z}}$ are independent and identically
distributed real-valued random variables that have mean $0$ and a finite
variance, which -- without loss of generality -- is set equal to $1$. The
coefficients $\phi _{j}$ are assumed to satisfy $\sum_{j=0}^{\infty
}\left\vert \phi _{j}\right\vert <\infty $ and $\sum_{j=0}^{\infty }\phi
_{j}\neq 0$. Furthermore, $\varepsilon _{i}$ is supposed to have a density $%
q $ with respect to (w.r.t.) Lebesgue-measure. We note that under these
assumptions $x_{t}$ possesses a density w.r.t. Lebesgue-measure for every $%
t\geq 1$, and the same is true for $w_{t}$; cf.~Section 3.1 in P\"{o}tscher
(2004). Furthermore, the characteristic function $\psi $ of $\varepsilon
_{i} $ is assumed to satisfy%
\begin{equation}
\int_{-\infty }^{\infty }\left\vert \psi (s)\right\vert ^{\nu }ds<\infty
\label{char}
\end{equation}%
for some $1\leq \nu <\infty $. \emph{These assumptions will be maintained
throughout the paper.} They have been used in P\"{o}tscher (2004), while
stricter versions occur, e.g., in Park and Phillips (1999), de Jong (2004),
and de Jong and Whang (2005). A detailed discussion of the scope of
condition (\ref{char}) is given in P\"{o}tscher (2004), Section 3.1. In
particular, we recall from Lemma 3.1 in P\"{o}tscher (2004) that under the
maintained conditions of the present paper densities $h_{t}$ of $%
t^{-1/2}x_{t}$ exist such that for a suitable integer $t_{\ast }\geq 1$%
\begin{equation}
\sup_{t\geq t_{\ast }}\left\Vert h_{t}\right\Vert _{\infty }<\infty
\label{densitybound}
\end{equation}%
is satisfied, where $\left\Vert \cdot \right\Vert _{\infty }$ denotes the
supremum norm. In the following we set $\kappa =\sup_{t\geq t_{\ast
}}\left\Vert h_{t}\right\Vert _{\infty }$.

\subsection{Bounds on the Order of Magnitude of $\sum_{t=1}^{n}\left\vert
x_{t}\right\vert ^{-\protect\alpha }$}

We first consider the behavior of $\sum_{t=1}^{n}\left\vert x_{t}\right\vert
^{-\alpha }$. Note that under our assumptions this quantity is almost surely
well-defined and finite for every $\alpha \in \mathbb{R}$.\footnote{%
In particular, how, and if, we assign a value in the extended real line to $%
\left\vert x_{t}\right\vert ^{-\alpha }$ on the event $\left\{
x_{t}=0\right\} $ has no consequence for the results.} Recall that we are
mainly interested in the case $\alpha \geq 1$. While the next theorem
provides an upper bound on the order of magnitude, lower bounds are
discussed in Remarks \ref{rem2} and \ref{rem8} below.

\begin{theorem}
\label{mainresult}%
\begin{equation*}
\sum_{t=1}^{n}\left\vert x_{t}\right\vert ^{-\alpha }=\left\{ 
\begin{array}{cc}
O_{\Pr }(n^{\alpha /2}) & \text{if \ \ }\alpha >1 \\ 
O_{\Pr }(n^{1/2}\log n) & \text{if \ \ }\alpha =1 \\ 
O_{\Pr }(n^{1-\alpha /2}) & \text{if \ \ }-2\leq \alpha <1.%
\end{array}%
\right.
\end{equation*}
\end{theorem}

\begin{proof}
Suppose first that $\alpha \geq 0$ holds. Since $\sum_{t=1}^{t_{\ast
}-1}\left\vert x_{t}\right\vert ^{-\alpha }$ is almost surely real-valued it
suffices to prove the result for $\sum_{t=t_{\ast }}^{n}\left\vert
x_{t}\right\vert ^{-\alpha }$. For $0<\delta <1$ we have almost surely%
\begin{eqnarray*}
\sum_{t=t_{\ast }}^{n}\left\vert x_{t}\right\vert ^{-\alpha }
&=&\sum_{t=t_{\ast }}^{n}\left\vert x_{t}\right\vert ^{-\alpha }\boldsymbol{1%
}\left( \left\vert t^{-1/2}x_{t}\right\vert >\delta /(nt)^{1/2}\right) \\
&&+\sum_{t=t_{\ast }}^{n}\left\vert x_{t}\right\vert ^{-\alpha }\boldsymbol{1%
}\left( \left\vert t^{-1/2}x_{t}\right\vert \leq \delta /(nt)^{1/2}\right) \\
&=&Q_{n}(\delta )+R_{n}(\delta )
\end{eqnarray*}%
where $t_{\ast }$ is as in (\ref{densitybound}) and $n\geq t_{\ast }$. First
consider $R_{n}(\delta )$: Set 
\begin{equation*}
S_{n}(\delta )=\bigcup_{t=t_{\ast }}^{n}\left\{ \left\vert
t^{-1/2}x_{t}\right\vert \leq \delta /(nt)^{1/2}\right\} .
\end{equation*}%
Observe that $\left\{ R_{n}(\delta )>0\right\} =S_{n}(\delta )$ up to
null-sets and%
\begin{eqnarray*}
\Pr \left( R_{n}(\delta )>0\right) &=&\Pr \left( S_{n}(\delta )\right) \
\leq \ \sum_{t=t_{\ast }}^{n}\Pr \left( \left\vert t^{-1/2}x_{t}\right\vert
\leq \delta /(nt)^{1/2}\right) \\
&=&\sum_{t=t_{\ast }}^{n}\int_{-\delta /(nt)^{1/2}}^{\delta
/(nt)^{1/2}}h_{t}(z)dz\ \leq \ 2\kappa \delta n^{-1/2}\sum_{t=t_{\ast
}}^{n}t^{-1/2} \\
&\leq &4\kappa \delta
\end{eqnarray*}%
holds for all $n\geq t_{\ast }$ in view of (\ref{densitybound}) using the
fact that $\sum_{t=t_{\ast }}^{n}t^{-1/2}\leq \sum_{t=1}^{n}t^{-1/2}\leq
2n^{1/2}$. Next we bound $Q_{n}(\delta )$: Observe that%
\begin{equation*}
EQ_{n}(\delta )=\sum_{t=t_{\ast }}^{n}t^{-\alpha /2}E\left( \left\vert
t^{-1/2}x_{t}\right\vert ^{-\alpha }\boldsymbol{1}(\left\vert
t^{-1/2}x_{t}\right\vert >\delta /(nt)^{1/2})\right) ,
\end{equation*}%
and that for $t\geq t_{\ast }$%
\begin{eqnarray*}
&&E\left( \left\vert t^{-1/2}x_{t}\right\vert ^{-\alpha }\boldsymbol{1}%
(\left\vert t^{-1/2}x_{t}\right\vert >\delta /(nt)^{1/2})\right) \\
&=&E\left( \left\vert t^{-1/2}x_{t}\right\vert ^{-\alpha }\boldsymbol{1}%
(1>\left\vert t^{-1/2}x_{t}\right\vert >\delta /(nt)^{1/2})\right) \\
&&+E\left( \left\vert t^{-1/2}x_{t}\right\vert ^{-\alpha }\boldsymbol{1}%
(\left\vert t^{-1/2}x_{t}\right\vert \geq 1)\right) \\
&\leq &\int_{\delta /(nt)^{1/2}<\left\vert z\right\vert <1}\left\vert
z\right\vert ^{-\alpha }h_{t}(z)dz\ +\ 1\ \leq \ 2\kappa \int_{\delta
/(nt)^{1/2}}^{1}z^{-\alpha }dz\ +\ 1 \\
&\leq &\left\{ 
\begin{array}{cc}
1+2\kappa (\alpha -1)^{-1}\delta ^{1-\alpha }(nt)^{(\alpha -1)/2} & \text{if
\ \ }\alpha >1 \\ 
1+2\kappa \log \left( \delta ^{-1}\right) +2\kappa \log \left( \left(
nt\right) ^{1/2}\right) & \text{if \ \ }\alpha =1 \\ 
1+2\kappa (1-\alpha )^{-1} & \text{if \ \ }0\leq \alpha <1.%
\end{array}%
\right.
\end{eqnarray*}%
Consequently, for $n\geq \max (t_{\ast },3)$ we have%
\begin{eqnarray*}
E(Q_{n}(\delta )) &\leq &\left\{ 
\begin{array}{cc}
\left( 1+2\kappa (\alpha -1)^{-1}\delta ^{1-\alpha }\right) n^{\left( \alpha
-1\right) /2}\sum_{t=t_{\ast }}^{n}t^{-1/2} & \text{if \ \ }\alpha >1 \\ 
\left( 1+2\kappa +2\kappa \log \left( \delta ^{-1}\right) \right) \left(
\log n\right) \sum_{t=t_{\ast }}^{n}t^{-1/2} & \text{if \ \ }\alpha =1 \\ 
\left( 1+2\kappa (1-\alpha )^{-1}\right) \sum_{t=t_{\ast }}^{n}t^{-\alpha /2}
& \text{if \ \ }0\leq \alpha <1.%
\end{array}%
\right. \\
&\leq &\left\{ 
\begin{array}{cc}
c(\alpha ,\delta ,\kappa )n^{\alpha /2} & \text{if \ \ }\alpha >1 \\ 
c(1,\delta ,\kappa )n^{1/2}\log n & \text{if \ \ }\alpha =1 \\ 
c(\alpha ,\delta ,\kappa )n^{1-\alpha /2} & \text{if \ \ }0\leq \alpha <1.%
\end{array}%
\right.
\end{eqnarray*}%
where $c(\alpha ,\delta ,\kappa )$ are positive finite constants.

Now, for arbitrary $\varepsilon >0$ choose $\delta (\varepsilon )$
satisfying $0<\delta (\varepsilon )<\min (1,\varepsilon /(8\kappa ))$. Then
choose $M=M(\varepsilon ,\alpha ,\kappa )>0$ large enough to satisfy%
\begin{equation*}
M>4\varepsilon ^{-1}c(\alpha ,\delta (\varepsilon ),\kappa ).
\end{equation*}%
Then, with $d_{n}=n^{\alpha /2}$ in case $\alpha >1$, $d_{n}=n^{1/2}\log n$
in case $\alpha =1$, and $d_{n}=n^{1-\alpha /2}$ in case $0\leq \alpha <1$,
we obtain using Markov's inequality%
\begin{eqnarray*}
&&\Pr \left( d_{n}^{-1}\sum_{t=t_{\ast }}^{n}\left\vert x_{t}\right\vert
^{-\alpha }>M\right) \\
&\leq &\Pr \left( d_{n}^{-1}Q_{n}(\delta (\varepsilon ))>M/2\right) \ +\ \Pr
\left( d_{n}^{-1}R_{n}(\delta (\varepsilon ))>M/2\right) \\
&\leq &2d_{n}^{-1}EQ_{n}(\delta (\varepsilon ))/M\ +\ \Pr \left(
R_{n}(\delta (\varepsilon ))>0\right) \ <\ \varepsilon
\end{eqnarray*}%
for all $n\geq \max (t_{\ast },3)$. Since $\sum_{t=t_{\ast }}^{n}\left\vert
x_{t}\right\vert ^{-\alpha }$ is almost surely real-valued for all $n\geq
t_{\ast }$, this completes the proof in case $\alpha \geq 0$.

Suppose next that $-2\leq \alpha <0$ holds. Observe first that%
\begin{equation}
\sum_{t=1}^{n}\left\vert x_{t}\right\vert ^{-\alpha }\leq \max \left(
1,2^{-\alpha -1}\right) \left( \sum_{t=1}^{n}\left\vert
x_{t}-x_{0}\right\vert ^{-\alpha }+n\left\vert x_{0}\right\vert ^{-\alpha
}\right) .  \label{smpl}
\end{equation}%
By Lyapunov's inequality and noting that $E\left( x_{t}-x_{0}\right) ^{2}$
is of the exact order $t$ (since $w_{t}$ is a linear process with absolutely
summable coefficients satisfying $\sum_{j=0}^{\infty }\phi _{j}\neq 0$) we
have 
\begin{equation*}
E\sum_{t=1}^{n}\left\vert x_{t}-x_{0}\right\vert ^{-\alpha }\leq
c\sum_{t=1}^{n}t^{-\alpha /2}=O(n^{1-\alpha /2})
\end{equation*}%
for some finite constant $c$. But then an application of Markov's inequality
gives $\sum_{t=1}^{n}\left\vert x_{t}-x_{0}\right\vert ^{-\alpha }=O_{\Pr
}(n^{1-\alpha /2})$. Together with (\ref{smpl}) this establishes the claim.
\end{proof}

\begin{remark}
\normalfont(i) The proof of Theorem \ref{mainresult} in the previous version
of this paper (dated January 2011) is incorrect. For a discussion of the
errors and an alternative proof see the supplementary notes available on my
webpage.

(ii) Remark 6 in the January 2011 version of this paper insinuated that
there is a contradiction between Theorem 1 and results in de Jong and Whang
(2005). However, the argument put forward in this remark is invalid as there
is an elementary sign-mistake in the inequality presented in that remark.
Hence, this remark is completely invalid and I owe apologies to de Jong and
Whang.
\end{remark}

\begin{remark}
\label{rem0}\normalfont(i) For values of $\alpha $ such that $x^{-\alpha }$
is well-defined for every $x$ except possibly for $x=0$, the quantity $%
\sum_{t=1}^{n}x_{t}^{-\alpha }$ is almost surely well-defined and
real-valued. By the triangle inequality Theorem \ref{mainresult} applies
also to $\sum_{t=1}^{n}x_{t}^{-\alpha }$.

(ii) Not surprisingly, the expectation of $\sum_{t=1}^{n}\left\vert
x_{t}\right\vert ^{-\alpha }$ will typically be infinite in the case $\alpha
\geq 1$ (e.g., if the density of $x_{t}$ is bounded from below in a
neighborhood of zero as is the case if $x_{t}$ is Gaussian). The expectation
can, however, also be infinite in other cases (e.g., if $\alpha <-2$ and
moments of $x_{t}$ of order $-\alpha $ do not exist).
\end{remark}

\begin{remark}
\label{negative}\normalfont(i) It follows from Remark \ref{rem2} below that
the bound given for $-2\leq \alpha <0$ holds in fact for all $\alpha <0$
provided the additional condition $\sum_{j=0}^{\infty }j^{1/2}\left\vert
\phi _{j}\right\vert <\infty $ is satisfied. [The additional condition is
perhaps unnecessary, but we do not make any effort to remove it as the focus
in this paper is on the case $\alpha \geq 1$.]

(ii) If $Ex_{0}^{2}<\infty $ holds, then $Ex_{t}^{2}=E\left(
x_{t}-x_{0}\right) ^{2}+Ex_{0}^{2}$ is of the order $t$ and thus $%
E\left\vert x_{t}\right\vert ^{-\alpha }$ is at most of the order $%
t^{-\alpha /2}$ for $-2\leq \alpha <0$ by Lyapunov's inequality. This shows
that if $Ex_{0}^{2}<\infty $ holds the proof of Theorem \ref{mainresult} for
the case $-2\leq \alpha <0$ can be simplified.
\end{remark}

\begin{remark}
\label{rem2}\normalfont Suppose the stronger summability condition $%
\sum_{j=0}^{\infty }j^{1/2}\left\vert \phi _{j}\right\vert <\infty $ is
satisfied. Under this additional assumption more is known in case $-\infty
<\alpha <1$ than just the upper bound on the order of magnitude of $%
\sum_{t=1}^{n}\left\vert x_{t}\right\vert ^{-\alpha }$ given by Theorem \ref%
{mainresult}: If $-\infty <\alpha <1$ then%
\begin{equation}
n^{\alpha /2-1}\sum_{t=1}^{n}\left\vert x_{t}\right\vert ^{-\alpha }\overset{%
d}{\rightarrow }\left\vert \sigma \right\vert ^{-\alpha
}\int_{0}^{1}\left\vert W(s)\right\vert ^{-\alpha }ds  \label{pot}
\end{equation}%
for $n\rightarrow \infty $, with the limiting variable being positive with
probability one; as a consequence, $n^{1-\alpha /2}$ is the exact order of
magnitude in probability of $\sum_{t=1}^{n}\left\vert x_{t}\right\vert
^{-\alpha }$. Here $W$ is standard Brownian motion and $\sigma
=\sum_{j=0}^{\infty }\phi _{j}$, which is non-zero by assumption.\footnote{%
Clearly, $\sigma ^{2}$ is nothing else than the so-called long-run variance.}
Relation (\ref{pot}) follows from the first claim in Corollary 3.3 in P\"{o}%
tscher (2004), applied to the function $T$ given by $T(x)=\left\vert
x\right\vert ^{-\alpha }$ for $x\neq 0$ and $T(0)=0$, and from the
observation that $n^{\alpha /2-1}\sum_{t=1}^{b}\left\vert x_{t}\right\vert
^{-\alpha }\rightarrow 0$ as $n\rightarrow \infty $ for every fixed integer $%
b$. Note that $T$ is locally integrable since $\alpha <1$ and that $T$
satisfies $T(\lambda x)=\left\vert \lambda \right\vert ^{-\alpha }T(x)$ for
all $x\in \mathbb{R}$ and all $\lambda \neq 0$. Also note that the integral
in (\ref{pot}) is almost surely well-defined and finite (independently of
how one interprets $\left\vert W(s)\right\vert ^{-\alpha }$ for $W(s)=0$ in
case $\alpha >0$), cf.~(2.4) and Remark 2.1 in P\"{o}tscher (2004). [In the
case $\alpha \leq 0$, it is well-known that (\ref{pot}) holds even under
much weaker conditions than used here, cf.~Lemma A.1 in P\"{o}tscher (2004).
Since the emphasis in this paper is on positive $\alpha $, we make no
attempt to spell out these sharper and well-known results for $\alpha \leq 0$%
.]
\end{remark}

\begin{remark}
\label{rem8}\footnote{%
The lower bound results for $\alpha \geq 1$ given in this remark together
with the lower bound results for the case $-\infty <\alpha <1$ implied by
Remark \ref{rem2} provide an improvement over Proposition 6.4 in Park and
Phillips (1999) under weaker conditions.}\normalfont(i) We first provide a
lower bound in case $\alpha =1$. Given the additional assumption $%
\sum_{j=0}^{\infty }j^{1/2}\left\vert \phi _{j}\right\vert <\infty $, a
lower bound for the order of magnitude in probability of $%
\sum_{t=1}^{n}\left\vert x_{t}\right\vert ^{-1}$ is given by $n^{1/2}$, in
the sense that%
\begin{equation*}
\lim_{n\rightarrow \infty }\Pr \left( n^{-1/2}\sum_{t=1}^{n}\left\vert
x_{t}\right\vert ^{-1}>M\right) =1
\end{equation*}%
holds for every real $M$, i.e., $n^{-1/2}\sum_{t=1}^{n}\left\vert
x_{t}\right\vert ^{-1}\rightarrow \infty $ in probability. To see this, let $%
T_{k,1}(x)=\min (k,\left\vert x\right\vert ^{-1})$ for $k\in \mathbb{N}$
with the convention that $T_{k,1}(0)=k$. Then we have almost surely%
\begin{equation*}
n^{-1/2}\sum_{t=1}^{n}\left\vert x_{t}\right\vert
^{-1}=n^{-1}\sum_{t=1}^{n}\left\vert n^{-1/2}x_{t}\right\vert ^{-1}\geq
n^{-1}\sum_{t=1}^{n}T_{k,1}(n^{-1/2}x_{t})
\end{equation*}%
for every $k\in \mathbb{N}$. Furthermore, $n^{-1}%
\sum_{t=1}^{n}T_{k,1}(n^{-1/2}x_{t})$ converges in distribution to $%
\int_{0}^{1}T_{k,1}(\sigma W(s))ds$ by Corollary 3.4 in P\"{o}tscher (2004).%
\footnote{%
Since $T_{k,1}$ is continuous, this convergence in fact holds under weaker
conditions on the process $x_{t}$ then used here, cf. Lemma A.1 in P\"{o}%
tscher (2004).} Now, by Corollary 7.4 in Chung and Williams (1990) and the
monotone convergence theorem we have almost surely%
\begin{equation*}
\int_{0}^{1}T_{k,1}(\sigma W(s))ds=\int_{-\infty }^{\infty }T_{k,1}(\sigma
x)L(1,x)dx\rightarrow \left\vert \sigma \right\vert ^{-1}\int_{-\infty
}^{\infty }\left\vert x\right\vert ^{-1}L(1,x)dx=\infty
\end{equation*}%
for $k\rightarrow \infty $, where $L$ denotes standard Brownian local time.
The last equality in the above display follows since $L(1,0)>0$ almost
surely and $L(1,x)$ having almost surely continuous sample path together
imply that there exists a neighborhood $U$ of zero (that may depend on the
realization of $L(1,\cdot )$) such that $\inf_{x\in U}L(1,x)>0$ holds almost
surely. Note that the just established lower bound (established under the
stricter summability condition on $\phi _{j}$ imposed here) and the upper
bound given by Theorem \ref{mainresult} agree up to a logarithmic term and
in this sense are close to being sharp.

(ii) We next turn to the case $\alpha >1$ and show that the upper bound $%
n^{\alpha /2}$ on the order of magnitude is also a lower bound in the sense
that 
\begin{equation}
\lim_{\varepsilon \rightarrow 0,\varepsilon >0}\liminf_{n\rightarrow \infty
}\Pr \left( n^{-\alpha /2}\sum_{t=1}^{n}\left\vert x_{t}\right\vert
^{-\alpha }>\varepsilon \right) =1  \label{low}
\end{equation}%
holds: To this end let $\beta _{n}$ be a sequence satisfying $\beta
_{n}\rightarrow \infty $ and $n^{-1}\beta _{n}\rightarrow 0$ as $%
n\rightarrow \infty $. Then we have almost surely%
\begin{eqnarray*}
\left( n^{-1}\beta _{n}\right) ^{1-\alpha }n^{-\alpha
/2}\sum_{t=1}^{n}\left\vert x_{t}\right\vert ^{-\alpha }
&=&n^{-1}\sum_{t=1}^{n}\beta _{n}\left\vert \beta
_{n}n^{-1/2}x_{t}\right\vert ^{-\alpha } \\
&\geq &n^{-1}\sum_{t=1}^{n}\beta _{n}T_{k,\alpha }(\beta _{n}n^{-1/2}x_{t}),
\end{eqnarray*}%
where $T_{k,\alpha }(x)=\min (k,\left\vert x\right\vert ^{-\alpha })$ for $%
k\in \mathbb{N}$ with the convention that $T_{k,\alpha }(0)=k$. Note that $%
T_{k,\alpha }$ is Lebesgue-integrable (since $\alpha >1$) and bounded. The
version of Theorem 3 in Jeganathan (2004) given as Proposition \ref%
{jeganathan} in the Appendix below now shows that the right-hand side of the
above display converges in distribution to 
\begin{equation*}
\left\vert \sigma \right\vert ^{-1}\int_{-\infty }^{\infty }T_{k,\alpha
}(x)dxL(1,0).
\end{equation*}%
Since $L(1,0)>0$ almost surely and $\int_{-\infty }^{\infty }T_{k,\alpha
}(x)dx\rightarrow \infty $ for $k\rightarrow \infty $, it follows that 
\begin{equation*}
\lim_{n\rightarrow \infty }\Pr \left( \left( n^{-1}\beta _{n}\right)
^{1-\alpha }n^{-\alpha /2}\sum_{t=1}^{n}\left\vert x_{t}\right\vert
^{-\alpha }>M\right) =1
\end{equation*}%
holds for every real $M$, i.e., $\left( n^{-1}\beta _{n}\right) ^{1-\alpha
}n^{-\alpha /2}\sum_{t=1}^{n}\left\vert x_{t}\right\vert ^{-\alpha
}\rightarrow \infty $ in probability. Note that $\alpha >1$ and that this
result holds for \emph{every }sequence $\beta _{n}$ satisfying $\beta
_{n}\rightarrow \infty $ and $n^{-1}\beta _{n}\rightarrow 0$. A fortiori it
then holds for every sequence $\beta _{n}>0$ satisfying $n^{-1}\beta
_{n}\rightarrow 0$. Hence we have that $\eta _{n}n^{-\alpha
/2}\sum_{t=1}^{n}\left\vert x_{t}\right\vert ^{-\alpha }\rightarrow \infty $
in probability for \emph{every }sequence $\eta _{n}\rightarrow \infty $. By
Lemma \ref{lem} in the Appendix it follows that $n^{\alpha /2}$ is a lower
bound in the sense of (\ref{low}).
\end{remark}

\begin{remark}
\normalfont(i) All results above for $\sum_{t=1}^{n}\left\vert
x_{t}\right\vert ^{-\alpha }$ apply analogously to sums of the form $%
\sum_{t=a}^{n}\left\vert x_{t}\right\vert ^{-\alpha }$ for any (fixed)
integer $a>1$. [This follows since $\sum_{t=1}^{a-1}\left\vert
x_{t}\right\vert ^{-\alpha }$ is almost surely finite]

(ii) In case $\alpha \leq 0$ all results given above for $%
\sum_{t=1}^{n}\left\vert x_{t}\right\vert ^{-\alpha }$ carry over to $%
\sum_{t=0}^{n}\left\vert x_{t}\right\vert ^{-\alpha }$. For $\alpha >0$ this
is again so, provided the distribution of $x_{0}$ does not assign positive
mass to the point $0$; otherwise, $\sum_{t=0}^{n}\left\vert x_{t}\right\vert
^{-\alpha }$ is undefined on the event where $x_{0}=0$; if one chooses to
define $\left\vert x_{0}\right\vert ^{-\alpha }=\infty $ on this event, then
the above results clearly do not apply (except for the lower bound given in
Remark \ref{rem8} which then holds a fortiori).
\end{remark}

\subsection{Bounds on the Order of Magnitude of $\sum_{t=1}^{n}v_{t}\left%
\vert x_{t}\right\vert ^{-\protect\alpha }$}

We next illustrate how the above results can be used to derive upper bounds
on the order of magnitude of $\sum_{t=1}^{n}v_{t}\left\vert x_{t}\right\vert
^{-\alpha }$ where $v_{t}$ for $t\geq 1$ are random variables defined on the
same probability space as $x_{t}$. Note that this expression is almost
surely well-defined and finite for every $\alpha \in \mathbb{R}$.\footnote{%
In particular, how, and if, we assign a value in the extended real line to $%
v_{t}\left\vert x_{t}\right\vert ^{-\alpha }$ on the event $\left\{
x_{t}=0\right\} $ has no consequence for the results.} The leading case we
have in mind is $v_{t}=w_{t+1}^{k}$ where $k\in \mathbb{N}$. Applying the
Cauchy-Schwarz inequality gives almost surely%
\begin{equation*}
\left\vert \sum_{t=1}^{n}v_{t}\left\vert x_{t}\right\vert ^{-\alpha
}\right\vert \leq \left( \sum_{t=1}^{n}v_{t}^{2}\right) ^{1/2}\left(
\sum_{t=1}^{n}\left\vert x_{t}\right\vert ^{-2\alpha }\right) ^{1/2}.
\end{equation*}%
Hence, if $\sup_{t\geq 1}Ev_{t}^{2}<\infty $ (or more generally $%
\sum_{t=1}^{n}Ev_{t}^{2}=O(n)$) holds, we obtain from Theorem \ref%
{mainresult}%
\begin{equation}
\sum_{t=1}^{n}v_{t}\left\vert x_{t}\right\vert ^{-\alpha }=\left\{ 
\begin{array}{cc}
O_{\Pr }(n^{(\alpha +1)/2}) & \text{if \ \ }\alpha >1/2 \\ 
O_{\Pr }(n^{3/4}\left( \log n\right) ^{1/2}) & \text{if \ \ }\alpha =1/2 \\ 
O_{\Pr }(n^{1-\alpha /2}) & \text{if \ \ }-1\leq \alpha <1/2.%
\end{array}%
\right.  \label{simple}
\end{equation}%
Under the additional assumption $\sum_{j=0}^{\infty }j^{1/2}\left\vert \phi
_{j}\right\vert <\infty $ the bound $O_{\Pr }(n^{1-\alpha /2})$ in fact
holds also for $\alpha <-1$, cf. Remark \ref{rem2}. Variations of the above
bound can obviously be obtained by using H\"{o}lder's inequality.

\begin{remark}
\normalfont In the case $\alpha =0$ the problem reduces to determining the
order of $\sum_{t=1}^{n}v_{t}$, a problem to which this paper has nothing to
add to the literature. We only observe that in this case the above bound can
clearly be improved to $O_{\Pr }(n^{1/2})$ whenever $v_{t}$ satisfies a
central limit theorem (as is, e.g., the case if $v_{t}=w_{t+1}$), or
whenever $E\left( \sum_{t=1}^{n}v_{t}\right) ^{2}=O(n)$. The latter
condition is, e.g., satisfied if $v_{t}$ is mean-zero and weakly stationary
with absolutely summable covariance function, or if $v_{t}$ is a sequence of
uncorrelated mean-zero random variables satisfying $\sup_{t\geq
1}Ev_{t}^{2}<\infty $. We do not further comment on such improvements as
they are not related to the subject of the paper.
\end{remark}

We next provide improvements on the bound (\ref{simple}) under appropriate
assumptions on $v_{t}$. Note that the assumptions on $v_{t}$ in the
subsequent proposition are certainly satisfied if $v_{t}$ is independent of $%
x_{t}$ (or of $x_{t}-x_{0}$, respectively) for every $t\geq 1$ and the first
absolute moment of $v_{t}$ is bounded uniformly in $t$. In particular, these
assumptions are satisfied for the important special case $v_{t}=w_{t+1}^{k}$
provided that $\phi _{j}=0$ for all $j>0$ (implying that $w_{t}=\varepsilon
_{t}$) and that $E\left\vert \varepsilon _{t}\right\vert ^{k}$ is finite.%
\footnote{%
The condition that $\phi _{j}=0$ for all $j>0$ can of course be replaced by
the more general condition $\phi _{l}\neq 0$ for some $l\geq 0$ and $\phi
_{j}=0$ for all $j\neq l$. This equally applies to the discussion
immediately preceding Propositions \ref{LW} and \ref{ref}.}

\begin{proposition}
\label{indep}Suppose that in addition to the maintained assumptions we have
that $\sup_{t\geq 1}E(\left\vert v_{t}\right\vert )<\infty $ holds. Assume
further that $E(\left\vert v_{t}\right\vert \mid x_{t})=E(\left\vert
v_{t}\right\vert )$ almost surely holds for all $t\geq 1$ if $\alpha \geq 0$%
, and that $E(\left\vert v_{t}\right\vert \mid x_{t}-x_{0})=E(\left\vert
v_{t}\right\vert )$ almost surely holds for all $t\geq 1$ if $-2\leq \alpha
<0$. Then 
\begin{equation*}
\sum_{t=1}^{n}\left\vert v_{t}\right\vert \left\vert x_{t}\right\vert
^{-\alpha }=\left\{ 
\begin{array}{cc}
O_{\Pr }(n^{\alpha /2}) & \text{if \ \ }\alpha >1 \\ 
O_{\Pr }(n^{1/2}\log n) & \text{if \ \ }\alpha =1 \\ 
O_{\Pr }(n^{1-\alpha /2}) & \text{if \ \ }-2\leq \alpha <1.%
\end{array}%
\right.
\end{equation*}%
A fortiori the same bound then holds for $\sum_{t=1}^{n}v_{t}\left\vert
x_{t}\right\vert ^{-\alpha }$.
\end{proposition}

\begin{proof}
Suppose $\alpha \geq 0$. For the same reasons as given in the proof of
Theorem \ref{mainresult} it suffices to bound $\sum_{t=t^{\ast
}}^{n}\left\vert v_{t}\right\vert \left\vert x_{t}\right\vert ^{-\alpha }$.
Define for $0<\delta <1$ 
\begin{equation*}
Q_{n}^{\prime }(\delta )=\sum_{t=t_{\ast }}^{n}\left\vert v_{t}\right\vert
\left\vert x_{t}\right\vert ^{-\alpha }\boldsymbol{1}\left( \left\vert
t^{-1/2}x_{t}\right\vert >\delta /(nt)^{1/2}\right)
\end{equation*}%
and 
\begin{equation*}
R_{n}^{\prime }(\delta )=\sum_{t=t_{\ast }}^{n}\left\vert v_{t}\right\vert
\left\vert x_{t}\right\vert ^{-\alpha }\boldsymbol{1}\left( \left\vert
t^{-1/2}x_{t}\right\vert \leq \delta /(nt)^{1/2}\right) .
\end{equation*}%
Observe that now the event $\{R_{n}^{\prime }(\delta )>0\}$ is contained in $%
S_{n}(\delta )$ up to null-sets where $S_{n}(\delta )$ has been defined in
the proof of Theorem \ref{mainresult}. Hence, 
\begin{equation*}
\Pr \left( R_{n}^{\prime }(\delta )>0\right) \leq 4\kappa \delta
\end{equation*}%
as shown in the proof of Theorem \ref{mainresult}. Furthermore, since $%
\left\vert v_{t}\right\vert $ is integrable and $\left\vert x_{t}\right\vert
^{-\alpha }\boldsymbol{1}\left( \left\vert t^{-1/2}x_{t}\right\vert >\delta
/(nt)^{1/2}\right) $ is a bounded $x_{t}$-measurable random variable, the
law of iterated expectations and the assumptions on $v_{t}$ imply that%
\begin{equation*}
EQ_{n}^{\prime }(\delta )\leq \left( \sup_{t\geq 1}E(\left\vert
v_{t}\right\vert )\right) \sum_{t=t_{\ast }}^{n}t^{-\alpha /2}E\left(
\left\vert t^{-1/2}x_{t}\right\vert ^{-\alpha }\boldsymbol{1}(\left\vert
t^{-1/2}x_{t}\right\vert >\delta /(nt)^{1/2})\right)
\end{equation*}%
holds. The remainder of the proof is then identical to the proof of Theorem %
\ref{mainresult}. Next suppose $-2\leq \alpha <0$. Then%
\begin{equation}
\sum_{t=1}^{n}\left\vert v_{t}\right\vert \left\vert x_{t}\right\vert
^{-\alpha }\leq \max \left( 1,2^{-\alpha -1}\right) \left(
\sum_{t=1}^{n}\left\vert v_{t}\right\vert \left\vert x_{t}-x_{0}\right\vert
^{-\alpha }+\left\vert x_{0}\right\vert ^{-\alpha }\sum_{t=1}^{n}\left\vert
v_{t}\right\vert \right) .  \label{split}
\end{equation}%
Observe that the second sum on the right-hand side of the above display is $%
O_{\Pr }(n)$ by an application of Markov's inequality (since $E\left\vert
v_{t}\right\vert $ is uniformly bounded by assumption) and since $\left\vert
x_{0}\right\vert ^{-\alpha }$ is well-defined and real-valued. Furthermore,
since $\left\vert v_{t}\right\vert $ is integrable and $\left\vert
x_{t}-x_{0}\right\vert ^{-\alpha }$ is a nonnegative real-valued random
variable we may use the law of iterated expectations again (conditioning
being on $x_{t}-x_{0}$) to obtain that the expectation of the first sum in (%
\ref{split}) is bounded by 
\begin{equation*}
\left( \sup_{t\geq 1}E(\left\vert v_{t}\right\vert )\right)
\sum_{t=1}^{n}E\left( \left\vert x_{t}-x_{0}\right\vert ^{-\alpha }\right) .
\end{equation*}%
This bound is then further treated exactly as in the proof of Theorem \ref%
{mainresult}.
\end{proof}

\begin{remark}
\normalfont If $Ex_{0}^{2}<\infty $ is assumed, the condition $E(\left\vert
v_{t}\right\vert \mid x_{t}-x_{0})=E(\left\vert v_{t}\right\vert )$ almost
surely can be replaced by $E(\left\vert v_{t}\right\vert \mid
x_{t})=E(\left\vert v_{t}\right\vert )$ almost surely also in case $-2\leq
\alpha <0$. The proof then proceeds by directly bounding $%
E\sum_{t=1}^{n}\left\vert v_{t}\right\vert \left\vert x_{t}\right\vert
^{-\alpha }$ by $\left( \sup_{t\geq 1}E(\left\vert v_{t}\right\vert )\right)
\sum_{t=1}^{n}E\left( \left\vert x_{t}\right\vert ^{-\alpha }\right) $; cf.
Remark \ref{negative}(ii).
\end{remark}

We next turn to the case where $v_{t}$ is a martingale difference sequence.
The improvement over the bound (\ref{simple}) is obtained in this case by
observing that the sequence $\sum_{t=1}^{n}v_{t}\left\vert x_{t}\right\vert
^{-\alpha }$ is then a martingale transform and by combining Theorem \ref%
{mainresult} with results in Lai and Wei (1982). [Note that $%
\sum_{t=1}^{n}v_{t}\left\vert x_{t}\right\vert ^{-\alpha }$ will typically
not be a martingale as the first moment will in general not exist, cf.
Remark \ref{rem0}(ii); hence, martingale central limit theorems are not
applicable.] The assumptions in the subsequent proposition are in particular
satisfied in the important special case where $v_{t}=w_{t+1}$ and $\phi
_{j}=0$ for all $j>0$ (implying that $v_{t}=w_{t+1}=\varepsilon _{t+1}$) by
choosing $\mathcal{F}_{t}$ as the $\sigma $-field generated by $%
x_{t+1},\ldots ,x_{1}$ for $t\geq 0$.

\begin{proposition}
\label{LW}Suppose that in addition to the maintained assumptions we have
that $(v_{t})_{t\geq 1}$ is a martingale difference sequence with respect to
a filtration $(\mathcal{F}_{t})_{t\geq 0}$ such that $\sup_{t\geq 1}E\left(
v_{t}^{2}\mid \mathcal{F}_{t-1}\right) <\infty $ holds almost surely. Assume
further that $x_{t}$ is $\mathcal{F}_{t-1}$-measurable for every $t\geq 1$.

(a) Then%
\begin{equation*}
\sum_{t=1}^{n}v_{t}\left\vert x_{t}\right\vert ^{-\alpha }=\left\{ 
\begin{array}{cc}
o_{\Pr }(n^{\alpha /2}\left( \log n\right) ^{1/2+\tau }) & \text{if \ }%
\alpha >1/2 \\ 
o_{\Pr }(n^{1/4}(\log n)^{1+\tau }) & \text{if \ }\alpha =1/2 \\ 
o_{\Pr }(n^{(1-\alpha )/2}\left( \log n\right) ^{1/2+\tau }) & \text{if \ }%
-1\leq \alpha <1/2%
\end{array}%
\right.
\end{equation*}%
holds for every $\tau >0$. Under the additional assumption $%
\sum_{j=0}^{\infty }j^{1/2}\left\vert \phi _{j}\right\vert <\infty $ the
bound given for the range $-1\leq \alpha <1/2$ continues to hold for the
range $-\infty <\alpha <1/2$.

(b) 
\begin{equation*}
\sum_{t=1}^{n}v_{t}^{2}\left\vert x_{t}\right\vert ^{-\alpha }=\left\{ 
\begin{array}{cc}
o_{\Pr }\left( n^{\alpha /2+\tau }\right) & \text{if \ \ }\alpha \geq 1 \\ 
o_{\Pr }\left( n^{1-\alpha /2+\tau }\right) & \text{if \ \ }-2\leq \alpha <1%
\end{array}%
\right.
\end{equation*}%
holds for every $\tau >0$. Under the additional assumption $%
\sum_{j=0}^{\infty }j^{1/2}\left\vert \phi _{j}\right\vert <\infty $ the
bound given for the range $-2\leq \alpha <1$ continues to hold for the range 
$-\infty <\alpha <1$.
\end{proposition}

\begin{proof}
Since $\sum_{s=1}^{t}w_{s}$ is a (nondegenerate) recurrent random walk under
the assumptions of the proposition that is not of the lattice-type (as it
has uncountably many possible values in the sense of Chung (2001, Section
8.3) by Lebesgue's differentiation theorem), it visits every interval
infinitely often almost surely. From independence of $x_{0}$ and $%
(w_{s})_{s\geq 1}$ we may conclude that almost surely $\left\vert
x_{t}\right\vert $ falls into the interval $(1/2,3/2)$ infinitely often.
This shows that the sum $\sum_{t=1}^{n}\left\vert x_{t}\right\vert ^{-\alpha
}$ diverges almost surely for every value $\alpha \neq 0$, the divergence
being trivial in case $\alpha =0$. Now apply Lemma 2(iii) in Lai and Wei
(1982) to conclude that%
\begin{equation*}
\sum_{t=1}^{n}v_{t}\left\vert x_{t}\right\vert ^{-\alpha }=o\left( \left(
\sum_{t=1}^{n}\left\vert x_{t}\right\vert ^{-2\alpha }\right) ^{1/2}\left(
\log \sum_{t=1}^{n}\left\vert x_{t}\right\vert ^{-2\alpha }\right)
^{1/2+\theta }\right) \text{ \ }a.s.
\end{equation*}%
and%
\begin{equation*}
\sum_{t=1}^{n}v_{t}^{2}\left\vert x_{t}\right\vert ^{-\alpha }=o\left(
\left( \sum_{t=1}^{n}\left\vert x_{t}\right\vert ^{-\alpha }\right)
^{1+\theta }\right) \text{ \ }a.s.
\end{equation*}%
for every $\theta >0$. Apply Theorem \ref{mainresult} as well as Remark \ref%
{rem2} (applied to $2\alpha $ and $\alpha $, respectively) to complete the
proof.
\end{proof}

\begin{remark}
\normalfont If $\sup_{t\geq 1}E\left( \left\vert v_{t}\right\vert ^{\gamma
}\mid \mathcal{F}_{t-1}\right) <\infty $ almost surely holds for some $%
\gamma >2$, applying Corollary 2 in Lai and Wei (1982) yields the slightly
better bound 
\begin{equation*}
\sum_{t=1}^{n}v_{t}\left\vert x_{t}\right\vert ^{-\alpha }=\left\{ 
\begin{array}{cc}
O_{\Pr }(n^{\alpha /2}\left( \log n\right) ^{1/2}) & \text{if \ }\alpha >1/2
\\ 
O_{\Pr }(n^{1/4}\log n) & \text{if \ }\alpha =1/2 \\ 
O_{\Pr }(n^{(1-\alpha )/2}\left( \log n\right) ^{1/2}) & \text{if \ }-1\leq
\alpha <1/2,%
\end{array}%
\right.
\end{equation*}%
where under the additional condition $\sum_{j=0}^{\infty }j^{1/2}\left\vert
\phi _{j}\right\vert <\infty $ the bound for the range $-1\leq \alpha <1/2$
again continues to hold for $-\infty <\alpha <1/2$.
\end{remark}

In case the martingale difference sequence is square-integrable with a
nonrandom conditional variance the bound in Part (a) of the above
proposition can be somewhat improved. I owe this observation to a referee.
Note that the subsequent proposition in particular covers the important
special case $v_{t}=w_{t+1}=\varepsilon _{t+1}$ mentioned above.

\begin{proposition}
\label{ref}Suppose that in addition to the maintained assumptions we have
that $(v_{t})_{t\geq 1}$ is a martingale difference sequence with respect to
a filtration $(\mathcal{F}_{t})_{t\geq 0}$ such that $E\left( v_{t}^{2}\mid 
\mathcal{F}_{t-1}\right) =E\left( v_{t}^{2}\right) $ holds almost surely for
all $t\geq 1$ and such that $\sup_{t\geq 1}E\left( v_{t}^{2}\right) <\infty $%
. Assume further that $x_{t}$ is $\mathcal{F}_{t-1}$-measurable for every $%
t\geq 1$. For the case $-1\leq \alpha <0$ assume additionally $%
Ex_{0}^{2}<\infty $. Then%
\begin{equation*}
\sum_{t=1}^{n}v_{t}\left\vert x_{t}\right\vert ^{-\alpha }=\left\{ 
\begin{array}{cc}
O_{\Pr }(n^{\alpha /2}) & \text{if \ }\alpha >1/2 \\ 
O_{\Pr }(n^{1/4}\left( \log n\right) ^{1/2}) & \text{if \ }\alpha =1/2 \\ 
O_{\Pr }(n^{(1-\alpha )/2}) & \text{if \ }-1\leq \alpha <1/2%
\end{array}%
\right.
\end{equation*}%
holds.
\end{proposition}

\begin{proof}
Assume $\alpha \geq 0$ first. For the same reasons as given in the proof of
Theorem \ref{mainresult} it suffices to bound $\sum_{t=t^{\ast
}}^{n}v_{t}\left\vert x_{t}\right\vert ^{-\alpha }$. For $0<\delta <1$ write 
$\sum_{t=t^{\ast }}^{n}v_{t}\left\vert x_{t}\right\vert ^{-\alpha }$ as $%
Q_{n}^{\ast }(\delta )+R_{n}^{\ast }(\delta )$ where 
\begin{equation*}
Q_{n}^{\ast }(\delta )=\sum_{t=t_{\ast }}^{n}v_{t}\left\vert
x_{t}\right\vert ^{-\alpha }\boldsymbol{1}\left( \left\vert
t^{-1/2}x_{t}\right\vert >\delta /(nt)^{1/2}\right)
\end{equation*}%
and 
\begin{equation*}
R_{n}^{\ast }(\delta )=\sum_{t=t_{\ast }}^{n}v_{t}\left\vert
x_{t}\right\vert ^{-\alpha }\boldsymbol{1}\left( \left\vert
t^{-1/2}x_{t}\right\vert \leq \delta /(nt)^{1/2}\right) .
\end{equation*}%
Observe that $\left\{ \left\vert R_{n}^{\ast }(\delta )\right\vert
>0\right\} \subseteq S_{n}(\delta )$ up to null-sets, and hence $\Pr \left(
\left\vert R_{n}^{\ast }(\delta )\right\vert >0\right) \leq 4\kappa \delta $
as shown in the proof of Theorem \ref{mainresult}. Observe that the terms
making up $Q_{n}^{\ast }(\delta )$ have a finite second moment since the
factor multiplying $v_{t}$ is bounded in view of $\alpha \geq 0$. By the
martingale difference property of $v_{t}$, by the assumptions on its
conditional variance, and since $x_{t}$ is $\mathcal{F}_{t-1}$-measurable we
obtain arguing similarly as in the proof of Theorem \ref{mainresult} and
setting $c=\sup_{t\geq 1}E\left( v_{t}^{2}\right) $%
\begin{eqnarray*}
EQ_{n}^{\ast }(\delta )^{2} &=&\sum_{t=t_{\ast }}^{n}Ev_{t}^{2}E\left(
\left\vert x_{t}\right\vert ^{-2\alpha }\boldsymbol{1}\left( \left\vert
t^{-1/2}x_{t}\right\vert >\delta /(nt)^{1/2}\right) \right) \\
&\leq &c\sum_{t=t_{\ast }}^{n}t^{-\alpha }E\left( \left\vert
t^{-1/2}x_{t}\right\vert ^{-2\alpha }\boldsymbol{1}\left( 1>\left\vert
t^{-1/2}x_{t}\right\vert >\delta /(nt)^{1/2}\right) \right) \\
&&+c\sum_{t=t_{\ast }}^{n}t^{-\alpha }E\left( \left\vert
t^{-1/2}x_{t}\right\vert ^{-2\alpha }\boldsymbol{1}\left( \left\vert
t^{-1/2}x_{t}\right\vert \geq 1\right) \right) \\
&\leq &c\sum_{t=t_{\ast }}^{n}t^{-\alpha }\left( 2\kappa \int_{\delta
/(nt)^{1/2}}^{1}z^{-2\alpha }dz\ +\ 1\right) .
\end{eqnarray*}%
This gives the bound%
\begin{equation*}
EQ_{n}^{\ast }(\delta )^{2}=\left\{ 
\begin{array}{c}
O\left( n^{\alpha }\right) \text{ \ if }\alpha >1/2 \\ 
O\left( n^{1/2}\log n\right) \text{ \ if \ }\alpha =1/2 \\ 
O\left( n^{1-\alpha }\right) \text{ \ if }0\leq \alpha <1/2.%
\end{array}%
\right.
\end{equation*}%
An argument similar to the one in the proof of Theorem \ref{mainresult} then
completes the proof in the case $\alpha \geq 0$. Next consider the case $%
-1\leq \alpha <0$. Since $Ex_{0}^{2}<\infty $ is assumed, we have that $%
\left\vert x_{t}\right\vert ^{-\alpha }$ is square-integrable for $-1\leq
\alpha <0$. Since $v_{t}$ is square-integrable by assumption, it follows
that $v_{t}\left\vert x_{t}\right\vert ^{-\alpha }$ is integrable and hence
is a martingale difference sequence w.r.t. $(\mathcal{F}_{t})_{t\geq 0}$. In
fact, $v_{t}\left\vert x_{t}\right\vert ^{-\alpha }$ is even
square-integrable for $-1\leq \alpha <0$: since $v_{t}^{2}$ and $\left\vert
x_{t}\right\vert ^{-2\alpha }$ are nonnegative and integrable, the law of
iterated expectations and the assumptions imply%
\begin{equation*}
E\left( v_{t}^{2}\left\vert x_{t}\right\vert ^{-2\alpha }\right) =E\left(
\left\vert x_{t}\right\vert ^{-2\alpha }E\left( v_{t}^{2}\mid \mathcal{F}%
_{t-1}\right) \right) =E\left( \left\vert x_{t}\right\vert ^{-2\alpha
}\right) E\left( v_{t}^{2}\right) <\infty .
\end{equation*}%
Now, $v_{t}\left\vert x_{t}\right\vert ^{-\alpha }$ being a
square-integrable martingale difference sequence implies that%
\begin{equation*}
E\left( \sum_{t=1}^{n}v_{t}\left\vert x_{t}\right\vert ^{-\alpha }\right)
^{2}=\sum_{t=1}^{n}Ev_{t}^{2}E\left( \left\vert x_{t}\right\vert ^{-2\alpha
}\right) \leq \sup_{t\geq 1}E\left( v_{t}^{2}\right)
c_{1}\sum_{t=1}^{n}t^{-\alpha }=O\left( n^{1-\alpha }\right)
\end{equation*}%
where we use the fact that $E\left\vert x_{t}\right\vert ^{-2\alpha }\leq
c_{1}t^{-\alpha }$ for a finite constant $c_{1}$ as shown in Remark \ref%
{negative}(ii). An application of Markov's inequality then proves the result.
\end{proof}

\begin{remark}
\normalfont We note that the bounds in Propositions \ref{indep} and \ref{ref}
are given only for $\alpha \geq -2$ or $\alpha \geq -1$, respectively. We
have not invested effort into extending the validity of these bounds beyond
this range. In the special case $v_{t}=w_{t+1}$ the bound for $%
\sum_{t=1}^{n}v_{t}\left\vert x_{t}\right\vert ^{-\alpha }$ is again $O_{\Pr
}(n^{(1-\alpha )/2})$ for $\alpha \leq -2$; this follows from Theorem 3.1 in
Ibragimov and Phillips (2008) which establishes distributional convergence
of $n^{(\alpha -1)/2}\sum_{t=1}^{n}w_{t+1}\left\vert x_{t}\right\vert
^{-\alpha }$. This theorem makes assumptions on the process $x_{t}$ that are
stronger in some dimensions (e.g., higher moment assumptions) but are weaker
in other respects (e.g., no assumption about existence of a density).
However, for $\alpha >-2$ (which includes the case of negative powers of
interest here) the results in Ibragimov and Phillips (2008) do not apply.
\end{remark}

\section*{\textbf{References}}

\ \ \ Berkes, I. \& L. Horvath (2006): "Convergence of Integral Functionals
of Stochastic Processes", Econometric Theory 22, 304--322.

Borodin, A.N. \& I.A. Ibragimov (1995): "Limit Theorems for Functionals of
Random Walks", Proceedings of the Steklov Institute of Mathematics 195(2).

Christopeit, N. (2009): "Weak Convergence of Nonlinear Transformations of
Integrated Processes: The Multivariate Case", Econometric Theory 25,
1180--1207.

Chung, K.L. (2001): A Course in Probability Theory. 3rd ed. Academic Press.

Chung, K.L. \& R.J. Williams (1990): Introduction to Stochastic Integration.
2nd ed. Birkh\"{a}user.

de Jong, R.M. (2004): "Addendum to 'Asymptotics for Nonlinear
Transformations of Integrated Time Series'", Econometric Theory 20, 627-635.

de Jong, R.M. \& C. Wang (2005): "Further Results on the Asymptotics for
Nonlinear Transformations of Integrated Time Series", Econometric Theory 21,
413-430.

Jeganathan, P. (2004): "Convergence of Functionals of Sums of R.V.s to Local
Times of Fractional Stable Motions", Annals of Probability 32, 1771--1795.

Lai, T.L. \& C.Z. Wei (1982): "Least Squares Estimates in Stochastic
Regression Models With Applications to Identification and Control of Dynamic
Systems", Annals of Statistics 10, 154-166.

Ibragimov, R. \& P.C.B. Phillips (2008): "Regression Asymptotics Using
Martingale Convergence Methods", Econometric Theory 24, 888-947.

Park, J.Y. \& P.C.B. Phillips (1999): "Asymptotics for Nonlinear
Transformations of Integrated Time Series", Econometric Theory 15, 269-298.

P\"{o}tscher, B.M. (2004): "Nonlinear Functionals and Convergence to
Brownian Motion: Beyond the Continuous Mapping Theorem", Econometric Theory
20, 1-22.

\appendix

\section{Appendix}

We first present a variant of Theorem 3 in Jeganathan (2004). If $x_{0}=0$,
the subsequent proposition follows immediately from Theorem 3 in Jeganathan
(2004). As we show in the proof below, for general $x_{0}$ the proposition
follows from that theorem combined with Remark 4 in Jeganathan (2004) plus a
conditioning argument. We also note that the assumptions on $x_{t}$ that we
maintain here are stronger than necessary and the proposition could also be
established under weaker conditions similar to the ones used in Jeganathan
(2004). We do not discuss such a more general result here.

\begin{proposition}
\label{jeganathan}Suppose $f$ is a Lebesgue-integrable real-valued function
on $\mathbb{R}$ that is bounded. Then, under the maintained assumptions on $%
x_{t}$, it holds that%
\begin{equation}
n^{-1}\sum_{t=1}^{n}\beta _{n}f(n^{-1/2}\beta _{n}x_{t})\overset{d}{%
\rightarrow }\left\vert \sigma \right\vert ^{-1}\left( \int_{-\infty
}^{\infty }f(y)dy\right) L(1,0)  \label{jegan}
\end{equation}%
for any sequence $\beta _{n}$ satisfying $\beta _{n}\rightarrow \infty $ and 
$n^{-1}\beta _{n}\rightarrow 0$. (Recall $\sigma =\sum_{j=0}^{\infty }\phi
_{j}\neq 0$.)
\end{proposition}

\begin{proof}
Without loss of generality we may assume that $\phi _{0}\neq 0$ (otherwise
shift the sequences $\phi _{j}$ and $\varepsilon _{i}$ appropriately). Set $%
\gamma _{n}=n^{1/2}h(n)\phi _{0}^{-1}\sigma $ as in Jeganathan (2004) with
positive $h(n)$, and note that $\gamma _{n}\neq 0$. From Proposition 1 in
Jeganathan (2004) we obtain that $\gamma _{n}^{-1}S_{n}=\gamma
_{n}^{-1}\sum_{t=1}^{n}w_{t}$ converges in distribution to $N(0,2)$. In view
of the central limit theorem for linear processes and the fact that $h(n)$
is positive, we conclude that $h(n)$ converges to $2^{-1/2}\left\vert \phi
_{0}\right\vert $. We also note that $\limfunc{sign}(\gamma _{n})=\limfunc{%
sign}(\phi _{0}^{-1}\sigma )$ is independent of $n$. Observe that%
\begin{equation}
n^{-1}\sum_{t=1}^{n}\beta _{n}f(n^{-1/2}\beta _{n}x_{t})=\left(
n^{1/2}/\left\vert \gamma _{n}\right\vert \right) n^{-1}\sum_{t=1}^{n}\bar{%
\beta}_{n}f^{\ast }(\gamma _{n}^{-1}\bar{\beta}_{n}x_{t})  \label{trans}
\end{equation}%
where $\bar{\beta}_{n}=n^{-1/2}\left\vert \gamma _{n}\right\vert \beta _{n}$
satisfies $\bar{\beta}_{n}\rightarrow \infty $ and $n^{-1}\bar{\beta}%
_{n}\rightarrow 0$ and where $f^{\ast }(y)=f(\limfunc{sign}(\phi
_{0}^{-1}\sigma )y)$.

Assume first that $x_{0}\equiv 0$. Then $x_{t}=S_{t}$ and since all
assumptions in Theorem 3(i) (or (ii)) in Jeganathan (2004) are satisfied, we
conclude from that theorem that the above expression converges weakly to $%
2^{1/2}\left\vert \sigma \right\vert ^{-1}\left( \int_{-\infty }^{\infty
}f^{\ast }(y)dy\right) \bar{L}(1,0)$ where $\bar{L}(1,0)$ is the local time
as defined in Jeganathan (2004). Since $\int_{-\infty }^{\infty }f^{\ast
}(y)dy=\int_{-\infty }^{\infty }f(y)dy$ and since $2^{1/2}\bar{L}(1,0)$ has
the same distribution as $L(1,0)$ the result follows in case $x_{0}\equiv 0$.

Next assume that $x_{0}\equiv c$, a constant not necessarily equal to zero.
By (\ref{trans}) it again suffices to show that $n^{-1}\sum_{t=1}^{n}\bar{%
\beta}_{n}f^{\ast }(\gamma _{n}^{-1}\bar{\beta}_{n}x_{t})=n^{-1}%
\sum_{t=1}^{n}f_{n}^{\ast }(\gamma _{n}^{-1}S_{t})$ converges to $\left(
\int_{-\infty }^{\infty }f^{\ast }(y)dy\right) \bar{L}(1,0)$ weakly, where $%
f_{n}^{\ast }(y)=\bar{\beta}_{n}f^{\ast }(\bar{\beta}_{n}(y+c\gamma
_{n}^{-1}))$. But, under the maintained assumptions on $x_{t}$, this follows
from the extension of Theorem 3 discussed in Remark 4 in Jeganathan (2004)
if we can verify the subsequent conditions for $f_{n}^{\ast }$ (we may
assume without loss of generality that $\bar{\beta}_{n}>0$ for all $n$):

(i) By change of variables and the integrability assumption on $f$ we have 
\begin{equation*}
\sup_{n}\int_{-\infty }^{\infty }\left\vert f_{n}^{\ast }(y)\right\vert
dy=\sup_{n}\int_{-\infty }^{\infty }\left\vert \bar{\beta}_{n}f^{\ast }(\bar{%
\beta}_{n}(y+c\gamma _{n}^{-1}))\right\vert dy=\int_{-\infty }^{\infty
}\left\vert f(y)\right\vert dy<\infty .
\end{equation*}

(ii) Correcting a typo in Jeganathan (2004), we have to show that 
\begin{equation*}
\limsup_{n}n^{-1}\int_{-\infty }^{\infty }\left\vert f_{n}^{\ast
}(y)\right\vert ^{2}dy=0.
\end{equation*}%
Note that the left-hand side can be written as%
\begin{equation*}
\limsup_{n}n^{-1}\int_{-\infty }^{\infty }\left\vert \bar{\beta}_{n}f^{\ast
}(\bar{\beta}_{n}(y+c\gamma _{n}^{-1}))\right\vert ^{2}dy=\limsup_{n}n^{-1}%
\bar{\beta}_{n}\int_{-\infty }^{\infty }\left\vert f(y)\right\vert ^{2}dy
\end{equation*}%
by a change of variables and the definition of $f^{\ast }$. But this is zero
since $n^{-1}\beta _{n}\rightarrow 0$ by assumption and since the integral
is finite ($f$ is quadratically integrable since it is integrable and
bounded).

(iii) Again by a change of variables%
\begin{equation*}
\lim_{d\rightarrow \infty }\sup_{n}\int_{\left\vert y\right\vert \geq
d}\left\vert f_{n}^{\ast }(y)\right\vert dy=\lim_{d\rightarrow \infty
}\sup_{n}\int_{\left\vert \bar{\beta}_{n}^{-1}z-c\gamma _{n}^{-1}\right\vert
\geq d}\left\vert f^{\ast }(z)\right\vert dz.
\end{equation*}%
Since $f$ is integrable, the limit for $d\rightarrow \infty $ is zero for
each integral individually. Hence, it suffices to show that 
\begin{equation*}
\lim_{d\rightarrow \infty }\sup_{n\geq N}\int_{\left\vert y\right\vert \geq
d}\left\vert f_{n}^{\ast }(y)\right\vert dy=\lim_{d\rightarrow \infty
}\sup_{n\geq N}\int_{\left\vert \bar{\beta}_{n}^{-1}z-c\gamma
_{n}^{-1}\right\vert \geq d}\left\vert f^{\ast }(z)\right\vert dz=0
\end{equation*}%
for a suitable $N$. Choose $N$ such that $\bar{\beta}_{n}>1$ and $\left\vert
c\gamma _{n}^{-1}\right\vert \leq 1$ holds for $n\geq N$. Then we have for $%
d>2$%
\begin{equation}
\sup_{n\geq N}\int_{\left\vert \bar{\beta}_{n}^{-1}z-c\gamma
_{n}^{-1}\right\vert \geq d}\left\vert f^{\ast }(z)\right\vert dz\leq
\int_{\left\vert z\right\vert \geq d/2}\left\vert f^{\ast }(z)\right\vert
dz=\int_{\left\vert z\right\vert \geq d/2}\left\vert f(z)\right\vert dz
\label{up}
\end{equation}%
since 
\begin{equation*}
\left\{ z:\left\vert \bar{\beta}_{n}^{-1}z-c\gamma _{n}^{-1}\right\vert \geq
d\right\} \subseteq \left\{ z:\left\vert z\right\vert \geq d/2\right\}
\end{equation*}%
for $n\geq N$ and $d>2$. The upper bound in (\ref{up}) now converges to zero
for $d\rightarrow \infty $ by integrability of $f$.

(iv) Define $F_{n}(y)$ as in Remark 4 in Jeganathan (2004). Then for $y\geq
0 $ we obtain%
\begin{equation*}
F_{n}(y)=\int_{0}^{y}\bar{\beta}_{n}f^{\ast }(\bar{\beta}_{n}(u+c\gamma
_{n}^{-1}))du=\int_{\bar{\beta}_{n}c\gamma _{n}^{-1}}^{\bar{\beta}%
_{n}(y+c\gamma _{n}^{-1})}f^{\ast }(z)dz,
\end{equation*}%
whereas for $y<0$ we obtain%
\begin{equation*}
F_{n}(y)=-\int_{y}^{0}\bar{\beta}_{n}f^{\ast }(\bar{\beta}_{n}(u+c\gamma
_{n}^{-1}))du=-\int_{\bar{\beta}_{n}(y+c\gamma _{n}^{-1})}^{\bar{\beta}%
_{n}c\gamma _{n}^{-1}}f^{\ast }(z)dz.
\end{equation*}%
It follows that 
\begin{equation*}
F_{n}(y)\rightarrow F(y)=\left\{ 
\begin{array}{cc}
\int_{0}^{\infty }f^{\ast }(z)dz & \text{if }y>0 \\ 
0 & \text{if }y=0 \\ 
-\int_{-\infty }^{0}f^{\ast }(z)dz & \text{if }y<0%
\end{array}%
\right.
\end{equation*}%
for every $y$. Observe that consequently $\int_{-\infty }^{\infty }\bar{L}%
(1,y)dF(y)=\left( \int_{-\infty }^{\infty }f^{\ast }(y)dy\right) \bar{L}%
(1,0) $.

(v) $\sup_{n,y}\bar{\beta}_{n}^{-1}\left\vert f_{n}^{\ast }(y)\right\vert
=\sup_{n,y}\left\vert f^{\ast }(\bar{\beta}_{n}(y+c\gamma
_{n}^{-1}))\right\vert =\sup_{y}\left\vert f^{\ast }(y)\right\vert
=\sup_{y}\left\vert f(y)\right\vert <\infty $ since $f$ is a bounded
function.

This proves that (\ref{jegan}) holds for arbitrary \emph{nonrandom} starting
values. If the starting value $x_{0}$ is random, we proceed as follows:%
\begin{eqnarray*}
&&\Pr \left( n^{-1}\sum_{t=1}^{n}\beta _{n}f(n^{-1/2}\beta _{n}x_{t})\leq
u\right) \\
&=&\int \Pr \left( n^{-1}\sum_{t=1}^{n}\beta _{n}f(n^{-1/2}\beta
_{n}x_{t})\leq u\mid x_{0}=c\right) dG(c) \\
&=&\int \Pr \left( n^{-1}\sum_{t=1}^{n}\beta _{n}f(n^{-1/2}\beta
_{n}(S_{t}+c)\leq u\right) dG(c)
\end{eqnarray*}%
where we have made use of independence of $x_{0}$ and $(S_{1},\ldots ,S_{n})$
and where $G$ denotes the distribution function of $x_{0}$. By what was
shown above, we have that $\Pr \left( n^{-1}\sum_{t=1}^{n}\beta
_{n}f(n^{-1/2}\beta _{n}(S_{t}+c)\leq u\right) $ converges to the
distribution function $\Pr \left( \left( \int_{-\infty }^{\infty
}f(y)dy\right) L(1,0)\leq u\right) $ for all continuity points of this
distribution function. Since this distribution function does not depend on $%
c $, we can conclude from dominated convergence that%
\begin{equation*}
\Pr \left( n^{-1}\sum_{t=1}^{n}\beta _{n}f(n^{-1/2}\beta _{n}x_{t})\leq
u\right) \rightarrow \Pr \left( \left( \int_{-\infty }^{\infty
}f(y)dy\right) L(1,0)\leq u\right)
\end{equation*}%
for all continuity points. This completes the proof.
\end{proof}

\begin{lemma}
\label{lem}Suppose $Y_{n}$ is a sequence of (real-valued or extended
real-valued) nonnegative random variables. Then the following are equivalent:

(i) $\eta _{n}Y_{n}\rightarrow \infty $ in probability as $n\rightarrow
\infty $ for \emph{every }sequence $\eta _{n}$ of real numbers satisfying $%
\eta _{n}\rightarrow \infty $.

(ii) $\lim_{\varepsilon \rightarrow 0,\varepsilon >0}\liminf_{n\rightarrow
\infty }\Pr \left( Y_{n}>\varepsilon \right) =1$.

(iii) $\liminf_{n\rightarrow \infty }\Pr \left( Y_{n}>\varepsilon
_{n}\right) =1$ for \emph{every }sequence of real numbers $\varepsilon
_{n}>0 $ satisfying $\varepsilon _{n}\rightarrow 0$.
\end{lemma}

\begin{proof}
We first show that (i) implies (iii): For given $\varepsilon _{n}>0$
satisfying $\varepsilon _{n}\rightarrow 0$ define $\eta _{n}=\varepsilon
_{n}^{-1}$. Clearly then $\eta _{n}\rightarrow \infty $ holds. From (i) we
then have that $\Pr \left( \eta _{n}Y_{n}>1\right) \rightarrow 1$ as $%
n\rightarrow \infty $. But this immediately translates into (iii).

Next we show that (iii) implies (i): Let $\eta _{n}\rightarrow \infty $ be a
given sequence and let $0<M<\infty $ be arbitrary. Define $\varepsilon
_{n}=M/\eta _{n}$ which is well-defined and positive for sufficiently large $%
n$ and satisfies $\varepsilon _{n}\rightarrow 0$. But then 
\begin{equation*}
1=\liminf_{n\rightarrow \infty }\Pr \left( Y_{n}>\varepsilon _{n}\right)
=\liminf_{n\rightarrow \infty }\Pr \left( \eta _{n}Y_{n}>M\right)
\end{equation*}%
holds as a consequence of (iii). Since $M$ was arbitrary, (i) follows.

That (ii) implies (iii) is obvious since for every $\varepsilon >0$ we have $%
\Pr \left( Y_{n}>\varepsilon _{n}\right) \geq \Pr \left( Y_{n}>\varepsilon
\right) $ for large $n$ since $\varepsilon _{n}\rightarrow 0$.

We finally show that (iii) implies (ii): Suppose (ii) does not hold. Then 
\begin{equation*}
\lim_{\varepsilon \rightarrow 0,\varepsilon >0}\liminf_{n\rightarrow \infty
}\Pr \left( Y_{n}>\varepsilon \right) <1
\end{equation*}%
must hold, noting that the outer limit exists due to monotonicity with
respect to $\varepsilon $. In particular,%
\begin{equation*}
\lim_{k\rightarrow \infty }\liminf_{n\rightarrow \infty }\Pr \left(
Y_{n}>1/k\right) <1
\end{equation*}%
must hold. Hence we can find a strictly increasing sequence $n_{k}$ of
integers diverging to infinity and a constant $c<1$ such that%
\begin{equation*}
\Pr \left( Y_{n_{k}}>1/k\right) <c<1
\end{equation*}%
holds for every $k\geq k_{0}$ for some sufficiently large $k_{0}$. For $%
n\geq n_{k_{0}}$ define $\varepsilon _{n}=1/k$ if $n_{k}\leq n<n_{k+1}$, and
set $\varepsilon _{n}=1$ for $n<n_{k_{0}}$. Then $\varepsilon _{n}>0$ and $%
\varepsilon _{n}\rightarrow 0$ for $n\rightarrow \infty $ holds. But 
\begin{equation*}
\liminf_{n\rightarrow \infty }\Pr \left( Y_{n}>\varepsilon _{n}\right) \leq
\liminf_{k\rightarrow \infty }\Pr \left( Y_{n_{k}}>\varepsilon
_{n_{k}}\right) \leq \liminf_{k\rightarrow \infty }\Pr \left(
Y_{n_{k}}>1/k\right) \leq c<1\text{,}
\end{equation*}%
showing that (iii) does not hold.
\end{proof}

\end{document}